\documentclass{article}

\usepackage{amssymb,latexsym,amsmath,epsfig,amsthm}
\usepackage{enumitem}
\usepackage{bbm}

\newtheorem{theorem}{Theorem}
\newtheorem{lemma}{Lemma}
\newtheorem{proposition}{Proposition}

\theoremstyle{definition}
\newtheorem{definition}{Definition}

\newtheorem{remark}{Remark}

\begin{document}

\title{Carmichael Numbers with a Specified Number of Prime Factors}
\author{Daniel Larsen \and Thomas Wright}
\date{}
\maketitle

\begin{center}
{\it To Carl Pomerance on his 80th birthday}
\end{center}

\begin{abstract}
For every sufficiently large integer $R$, there exists a Carmichael number with exactly $R$ prime factors.
\end{abstract}

\section{Introduction}

A natural number $n$ is said to be \emph{Carmichael} if for every integer $a$, $a^n \equiv a \pmod{n}$. Korselt's criterion \cite{K:1899} tells us that a square-free composite number $n$ is Carmichael if and only if $p-1 \mid n-1$ for every prime $p$ dividing $n$.

Alford, Granville, and Pomerance \cite{AGP:1994} proved in 1994 that there exist infinitely many Carmichael numbers. Moreover, it is known \cite{AGHS:2014} that for every $R$ between three and 10 billion, there exist Carmichael numbers with precisely $R$ prime factors. Granville and Pomerance \cite{GP:2002} have conjectured that for every $R \geq 3$, there are $x^{1/R + o(1)}$ Carmichael numbers less than $x$ with exactly $R$ prime factors, a conjecture that seems far out of reach of modern methods.

In this paper, we will prove the following weaker form of this conjecture.

\begin{theorem}\label{thm:main}
For all sufficiently large $R$, there exist Carmichael numbers with precisely $R$ prime factors.
\end{theorem}

Such a result was previously known only under the assumption of Dickson's prime $k$-tuples conjecture \cite{C:1935}.  Our result establishes unconditionally that there exist infinitely many Carmichael numbers with a prime number of prime factors, strengthening a result of the second author \cite{W:2024a} which showed this to be true under the assumption of a conjecture of Heath-Brown regarding the first prime in an arithmetic progression.

The traditional AGP method for producing infinite families of Carmichael numbers is based on constructing a set of primes $\mathcal{P}$ along with positive integers $k$ and $L$ such that:
\begin{itemize}
\item $k \mid p-1$ for every $p \in \mathcal{P}$
\item $\frac{p-1}{k} \mid L$ for every $p \in \mathcal{P}$ 
\item $\lambda(L)$, the maximum order of an element in $(\mathbb{Z}/L\mathbb{Z})^{\times}$, is small compared to $|\mathcal{P}|$
\item $k$ and $L$ are relatively prime
\end{itemize} One then studies the distribution of randomly chosen products of elements of $\mathcal{P}$ modulo $L$. Essentially because of the absence of large cyclic subgroups of $(\mathbb{Z}/L\mathbb{Z})^{\times}$, it can be shown that with probability at least $1/\phi(L) - \epsilon$, such a product is congruent to 1 modulo $L$. Let $\Pi$ be such a product. Then $\Pi \equiv 1 \pmod{L}$, but since each of its factors has the form $dk+1$ for some integer $d$, $\Pi$ is also congruent to $1 \pmod{k}$. Since $\gcd(L,k)=1$, $\Pi \equiv 1 \pmod{Lk}$. Finishing off the argument, every prime factor $p = dk+1$ of $\Pi$ satisfies $p-1 \mid dk \mid Lk \mid \Pi-1$. Hence $\Pi$ is a Carmichael number.

This proof relies on the size of $\mathcal{P}$ being large compared to $\lambda(L)$, but in fact, we can decompose $\mathcal{P}$ into many disjoint subsets, each of which is still large compared to $\lambda(L)$. It is easy to see that if we pick one Carmichael number obtained from each subset, we obtain a large set of Carmichael numbers such that all products obtained by multiplying distinct elements from this pool are themselves Carmichael. The basic point here is that the product of numbers congruent to 1 mod $L$ is itself 1 mod $L$. This story can be found in \cite{AGP:1994}, and one conclusion is that the set $\{\omega(n) : n \text{ Carmichael}\}$ has some additive structure.

Our argument begins with the observation that for such $\mathcal{P}$, we can produce long sequences of Carmichael numbers, each of which is formed from the previous one by multiplying in $h$ additional prime factors, for some positive integer $h$. For instance, $h$ could be the most frequent number of prime factors for the Carmichael numbers in our pool. The challenge then becomes showing that for arbitrary $a \bmod{h}$, we can construct Carmichael numbers from $\mathcal{P}$ with $a \bmod{h}$ prime factors. To do this, we will use methods similar to \cite{L:2025}. All that then remains is to show that we can fuse these two constructions to cover every number of the form $a + bh$ with $b$ in a reasonable range, which is to say, cover every number in a large interval. Repeating this construction with larger sets $\mathcal{P}$, we obtain overlapping sets of integers covered by our construction going out to infinity.

It will turn out that because $(\mathbb{Z}/L\mathbb{Z})^{\times}$ has a terrifyingly large number of subgroups of index 2, we will have some additional difficulties in producing Carmichael numbers with an odd number of prime factors. The issue of handling index 2 subgroups is a well-known obstruction in the study of Carmichael numbers - see, for instance, the difference between the lower bounds in \cite{M:2013} and \cite{W:2013}. However, our specific construction cannot use the techniques of \cite{W:2013} without taking $x$, the parameter which controls the size of the prime factors of our Carmichael numbers, infeasibly large in terms of the base parameter $y$. The solution is to instead fix a prime factor like 3 and try to arrange for the number of prime factors added on to produce a Carmichael number to be even, which can be done quite naturally.

We will begin by recalling the machinery of \cite{L:2025}, then stepping back temporarily to recall some key ideas underlying the machinery, and adapting them to suit our particular needs. We will then be able to prove that all sufficiently large even numbers are represented by values of $\omega$ on Carmichael numbers. Finally, we will reengage with the full machinery to complete the proof.
\section{Notation and Terminology}

We use a combination of Vinogradov notation and $O$-notation, at times preferring the flexibility of the latter, which allows for expressions such as $x^{1-o(1)}$, which represents a quantity larger than $x^{1-\epsilon}$ for any fixed $\epsilon > 0$, for sufficiently large $x$.

We say \emph{almost none} of the elements in a set $\mathcal{S}$ have a property if the number of elements with the property is $o(|\mathcal{S}|)$. Conversely, \emph{almost all} the elements have a property if almost none of the elements satisfy the negation of the property.

We let $\phi(n)$ represent Euler's totient function and $\lambda(n)$ represent Carmichael's lambda function, defined as the smallest positive $m$ for which $a^m\equiv 1\pmod n$ for all $a\in (\mathbb Z/n\mathbb Z)^\times$. We also define $\omega(n)$ to be the number of distinct prime factors of $n$.
\section{Establishing Equidistribution}
We begin by summarizing the construction which occupies the first seven sections of \cite{L:2025}.
\begin{proposition}\label{prop:main}
There exist positive integers $x$, $L_1$, $L_2$, $k_1$, $k_2$, $A_1$, $A_2$ and sets $\mathcal{P}_1$, $\mathcal{P}_2$ such that:

\medskip

\textbf{Integer properties:}
\begin{enumerate}[label=(\roman*)]
\item $k_1, k_2 \leq x$
\item $k_1$ and $k_2$ are divisible by $2$ but not by $3$ or $4$
\item $k_1k_2$ is relatively prime to $L_1L_2$
\item $L_1$ is relatively prime to $L_2$
\item $\gcd(\phi(k_i), \phi(L_j))$ is a power of $2$ for any choice of $i,j \in \{1,2\}$
\item $\gcd(\phi(L_1), \phi(L_2))$ is a power of $2$
\item \label{item:orderbound} $\max_{p \mid k_1k_2L_1L_2} \mathrm{ord}_{2}(p-1) = O(1)$
\item $3$ is a quadratic residue modulo $L_1$ and $L_2$
\item $\lambda(L_1)$ and $\lambda(L_2)$ are not divisible by $4$
\end{enumerate}

\textbf{Prime properties:}
\begin{enumerate}[label=(\roman*),resume]
\item Every element in $\mathcal{P}_1$ has the form $dk_1+1$ where $d$ divides $L_1$
\item Every element in $\mathcal{P}_2$ has the form $dk_2+1$ where $d$ divides $L_2$
\item $\lambda(L_1L_2)=|\mathcal{P}_1|^{o(1)}$ and $\lambda(L_1L_2)=|\mathcal{P}_2|^{o(1)}$
\end{enumerate}

\pagebreak[3]
\textbf{Carmichael properties:}
\begin{enumerate}[label=(\roman*),resume]
\item $A_1 \equiv 1 \pmod{L_1}$ and $A_1 \equiv 3^{-1} \pmod{L_2}$
\item $A_2 \equiv 1 \pmod{L_2}$ and $A_2 \equiv 3^{-1} \pmod{L_1}$
\item Every prime factor of $A_1$ has the form $dk_1+1$ where $d$ divides $L_1$
\item Every prime factor of $A_2$ has the form $dk_2+1$ where $d$ divides $L_2$
\item $A_1$ and $A_2$ each have an even number of prime factors
\item \label{item:oddorder} $3A_1$ has odd order modulo $k_2$ and $3A_2$ has odd order modulo $k_1$
\item $A_1$ and $A_2$ are relatively prime
\item $A_1A_2$ is relatively prime to $k_1k_2L_1L_2$
\item $A_1A_2$ is not divisible by any element of $\mathcal{P}_1$ or $\mathcal{P}_2$
\end{enumerate}

\textbf{Equidistribution properties:}
\begin{enumerate}[label=(\roman*),resume]
\item \label{item:charsum1} For every non-principal character $\chi$ modulo $k_2$, for any $z$, there are at least $|\mathcal{P}_1|/\log x$ elements $p \in \mathcal{P}_1$ for which $\chi(p) \neq z$
\item \label{item:charsum2} For every non-principal character $\chi$ modulo $k_1$, for any $z$, there are at least $|\mathcal{P}_2|/\log x$ elements $p \in \mathcal{P}_2$ for which $\chi(p) \neq z$
\item \label{item:prod1} For any set of $\lfloor k_2/(\log x \log\log x)\rfloor$ residues modulo $k_2$, almost none of the integers which can be written as a product of $10$ elements from $\mathcal{P}_1$ have modulo $k_2$ residue in that set
\item \label{item:prod2} For any set of $\lfloor k_1/(\log x \log\log x)\rfloor$ residues modulo $k_1$, almost none of the integers which can be written as a product of $10$ elements from $\mathcal{P}_2$ have modulo $k_1$ residue in that set
\end{enumerate}
\end{proposition}
In proving this, a major focus is on showing that the divisors of $L_1$ are distributed almost uniformly in the maximal proper subgroups of the Sylow subgroups of $(\mathbb{Z}/L_2\mathbb{Z})^{\times}$. The input to this result, codified in Lemmas 10 and 11 from \cite{L:2025}, is Property 7*. However, this property derives from Property 7, which tells us that it is equally true that the divisors of $L_1$ are distributed almost uniformly in the maximal proper subgroups of the Sylow subgroups of $(\mathbb{Z}/L_1\mathbb{Z})^{\times}$.

To understand where this construction comes from, it is helpful to take a step backward, which will also allow us to impose the stronger uniformity we require, as in the previous paragraph. It is convenient to introduce the following definition.

\begin{definition}
A sequence $a_1,\ldots, a_k$ of positive numbers is said to be \emph{logarithmically spread} if for every $i < k$, $1+1/1000 \leq \log a_{i+1}/\log a_i \leq 10$.
\end{definition}

Now we have the following proposition, combining Proposition 3 and Lemma 10 of \cite{L:2025} and taking $y_1=y^{\rho}$, $y_2=y^{\theta}$, $y_3=y^{\theta(1+\iota)}$, $y_4=y^{\theta+\rho-\iota}$, $y_5=y$, and $w = \lfloor y^{\theta + \iota\rho} \rfloor$. (Recall that in that paper, $\iota$ is a small constant (for concreteness, let us specify $\iota = 1/1000$), $\rho = 1/24 - 2\iota$, and $\theta = 1/6 - 2\iota$.)

\begin{proposition}
For every large integer $y_5$, there exist positive numbers $y_1$, $y_2$, $y_3$, and $y_4$ as well as a positive integer $w$ and a positive square-free integer $L$ such that the following hold:
\begin{itemize}
\item The sequence $y_1, y_2, w, y_3, y_4, y_5$ is logarithmically spread
\item $y_2 = y_5^{1/6 - 1/500}$
\item $\omega(L) \geq y_4$
\item $q-1$ is square-free for every prime $q$ dividing $L$
\item Taking $L_{(p)}$ to be the product of those primes $q$ with $p \mid q-1$ and $q \mid L$, $L_{(p)} > 1$ only for $p \in [y_3,y_5]$ or $p = 2$
\item For every $p > 2$, $\omega(L_{(p)}) \leq y_1$
\item The number of primes $p$ with $L_{(p)} > 1$ is $O(y_2)$
\item $3$ is a quadratic residue modulo $L$
\item Taking $\mathcal{D}$ to be the set of divisors of $L$ with $w$ prime factors,
\begin{equation}\label{eq:uniform}
\#\{d\in \mathcal{D}: d\equiv a \pmod{L_{(p)}}\}=\frac{|\mathcal{D}|}{\phi(L_{(p)})}\big(1+o(1)\big)
\end{equation}
for each residue class $a\in (\mathbb{Z}/L_{(p)}\mathbb{Z})^{\times}$ for every odd prime $p$ with $L_{(p)}>1$.
\end{itemize}
\end{proposition}

\begin{remark}
Roughly speaking, this proposition demonstrates how $L_1$ and $L_2$ are constructed in Proposition \ref{prop:main}. From there, $k_i$ are chosen to make 
\[\mathcal{P}_i = \{dk_i+1 \in \mathbb{P} : d \mid L_i\}\]
large.
\end{remark}

Note that
\begin{equation}\label{eq:product}
\sum_{q \mid L}\frac{1}{q} \leq \sum_{\substack{p > 2 \\ L_{(p)}>1}}\sum_{\substack{q \equiv 1 \pmod{p} \\ q \in [y_3,y_5]}}\frac{1}{q} = \sum_{\substack{p > 2 \\ L_{(p)}>1}}\frac{O(1)}{p} = \frac{O(y_2)}{y_3} = o(1),
\end{equation}
which in particular means that almost all integers less than $L$ are units modulo $L$.

Let $s=y_1^2$. Let $k$ be a fixed integer relatively prime to $L$. Fix an odd prime $p$ with $L_{(p)}>1$, let $G= (\mathbb{Z}/L_{(p)}\mathbb{Z})^{\times}$, and let $G'$ be a coset of index $p$. Consider the probability that $d_1k+1,\ldots,d_sk+1\in G'$ where each $d_i$ is chosen uniformly and independently from $\mathcal{D}$. This of course is equal to the $s$th power of the probability that $(dk+1)$ is contained in $G'$ for a single $d$ chosen uniformly from $\mathcal{D}$, which by Equations \eqref{eq:uniform} and \eqref{eq:product} is $\frac{1}{p}\big(1+o(1)\big)$. Thus the overall probability is bounded by 
$$\left(\frac{1}{p}\big(1+o(1)\big)\right)^s\leq y_3^{-s(1-o(1))}.$$

Note that $G$ contains $p \cdot \frac{p^{\mathrm{rank}_p(G)}-1}{p-1}$ cosets of index $p$. Recall that $\omega(L_{(p)}) \leq y_1$. Thus, the probability that any $G'$ exists is bounded by 
\begin{align*}
&\underbrace{y_3^{-s(1-o(1))}}_{\substack{\text{probability for} \\ \text{single coset}}} \cdot \underbrace{y_5^{O(y_1)}}_{\substack{\text{number of cosets} \\ \text{for fixed $p$}}} \cdot \underbrace{y_5}_{\substack{\text{number of} \\ \text{choices of $p$}}}.
\end{align*}
That is, the probability that $d_1k+1,\ldots,d_sk+1$ all lie in a proper coset of $G$ of odd index is bounded by 
\begin{equation}\label{eq:probbound}
y_3^{-s(1-o(1))}.
\end{equation}

\section{Finding Primes}

Let $\pi(x;q,a)$ count the number of primes up to $x$ congruent to $a$ mod $q$. As in \cite{AGP:1994}, for $B = 1/3$, it holds for sufficiently large $x$ that 
$$\pi(dx;d,1) \geq \frac{\pi(dx)}{2\phi(d)}$$
for every $d \leq x^B$ not containing a divisor in some exceptional set $\mathcal{D}_B(x)$ of bounded size. Set $x=y_5^{\frac{w}{1/6-3/1000}}$. This choice is simply to maintain compatibility with \cite{L:2025}; the important thing is that every $d \in \mathcal{D}$ is less than $x^B$. Let $\mathcal{D}'$ be the subset of $\mathcal{D}$ consisting of those elements which are divisible by none of the elements in $\mathcal{D}_B(x)$. For fixed $d \in \mathcal{D}'$, let $\mathcal{K}_d$ be the set of $k < x$ with $dk+1$ prime. We then have 
$$|\mathcal{K}_d| = \pi(dx;d,1) > \frac{x}{3\log x}.$$
Indeed, 
$$\frac{\phi(d)}{d} \geq \frac{\phi(L)}{L} = 1-o(1)$$
by Equation \eqref{eq:product}. Now for each prime $q$ dividing $L$, the number of $k$ divisible by $q$ in $\mathcal{K}_d$ is 
$$\pi(dx;dq,1) < \frac{3dx}{\log x \cdot \phi(dq)} < \frac{4x}{\log x \cdot q}$$
by the Brun-Titchmarsh theorem. Therefore, the number of $k \in \mathcal{K}_d$ relatively prime to $L$ is greater than
$$\frac{x}{3\log x} - 4\sum_{q \mid L}\frac{x}{\log x \cdot q} > \frac{x}{4\log x}$$
by Equation \eqref{eq:product}.

Let $\widetilde{\mathcal{K}}$ be the set of $k \leq x$ for which there exists a proper odd order coset of $(\mathbb{Z}/L\mathbb{Z})^{\times}$ containing $dk+1$ for at least $|\mathcal{D}|/(\log x \log\log x)$ values of $d \in \mathcal{D}'$. Then there exist $|\widetilde{\mathcal{K}}| \cdot (|\mathcal{D}|/(\log x \log\log x))^s$ choices of $k \leq x$ and $d_1,\ldots,d_s \in \mathcal{D}'$ (where repeats are allowed and order matters) with all of $\{d_ik+1\}_i$ in some proper coset. However, by Equation \eqref{eq:probbound}, there are at most $x|\mathcal{D}|^s/y_3^{s(1-o(1))}$ such choices with that property. Hence 
$$|\widetilde{\mathcal{K}}| < \frac{x(\log x \log\log x)^s}{y_3^{s(1-o(1))}}=x\left(\frac{\log x \log\log x}{y_3^{1-o(1)}}\right)^s$$
and since $\log x \ll w\log y_5$, this quantity is far smaller than, for instance, $x/\log^2 x$.

Let $\mathcal{K}'_d = \mathcal{K}_d \setminus \widetilde{\mathcal{K}}$. We have shown that $|\mathcal{K}'_d| \geq x/(4\log x)$. By the pigeonhole principle, there exists a specific value $k \in \mathcal{K}'_d$ such that there are at least $|\mathcal{D}|/(4\log x)$ primes of the form $dk+1$ with $d \in \mathcal{D}'$. Let $\mathcal{P}$ be this set.

Note that this construction, performed twice with separate values of $L$ and $k$ and some careful bookkeeping which can be found in \cite{L:2025}, is enough to recover most of Proposition \ref{prop:main}. More importantly, Section 7 of \cite{L:2025} shows that the equidistribution conditions of Proposition \ref{prop:main} hold generically, so we essentially obtain them for free. It remains to discuss the construction of $A_1$ and $A_2$, which we will take up in the next section.

\section{Producing Carmichael Numbers}

The Davenport constant for a finite abelian group $G$ is the smallest integer $n$ for which any sequence of $n$ elements of $G$ contains a non-empty subsequence whose product is equal to the identity.

We will need the following bound about the Davenport constant, adapted from Theorem 2 of \cite{AGP:1994}.

\begin{proposition}\label{prop:davenport2}
Let $n$ be the Davenport constant of $(\mathbb{Z}/L\mathbb{Z})^{\times}$. Then
$$n \leq \lambda(L)(1 + \log(|L|/\lambda(L))) + 1.$$
\end{proposition}

We will also need Proposition 7 from \cite{L:2025}.

\begin{proposition}\label{alteredAGPprop}
Let $G$ be a finite abelian group, let $g\in G$, and let $n$ be such that for any sequence of $n$ elements of $G$, there is a non-empty subsequence whose product is 1. Suppose $r>t>n$. Suppose $S$ is a sequence of length $r$ whose terms multiply to $g$. Then there are at least $\binom{r}{t}\binom{r}{n}^{-1}$ subsequences of $S$ of length at most $t$ whose product is $g$.
\end{proposition}

Since with probability 1, for $n$ randomly chosen primes, there is a sub-product congruent to 1 modulo $L$, there exists $h_0 \leq n$ such that with probability at least $1/n$, for any randomly chosen set of primes taken from $\mathcal{P}$, there is a sub-product of length $h_0$ congruent to 1 modulo $L$.

Let $r = y_5^{\sqrt{y_2 w}}$. Note that
$$\lambda(L) = \prod_{p: L_{(p)} > 1} p = y_5^{O(y_2)}$$
whereas
$$|\mathcal{P}| \gg \binom{\omega(L)}{w} \cdot (\log x)^{-1} > y_5^{\epsilon w}$$
for some absolute constant $\epsilon > 0$. Consequently,
\begin{equation}\label{eq:Pbound}
\log \lambda(L) / \log r = o(1), \quad \log r / \log |\mathcal{P}| = o(1), \quad \text{and } \log |\mathcal{P}| \gg \log x.
\end{equation}

\begin{proposition}\label{prop:Lresidue}
Let $a \bmod{h}$ be any arithmetic progression containing even numbers with $h \ll \lambda(L)^2$. For any quadratic residue $\square \bmod{L}$, and for almost all $r$-element subsets $\mathcal{P}^{\star} \subset \mathcal{P}$, there exists a positive integer $N$ with the following properties:
\begin{itemize}
\item $N > 1$
\item All prime factors of $N$ are in $\mathcal{P}^{\star}$
\item $N \equiv \square \pmod{L}$
\item $\omega(N) \equiv a \pmod{h}$
\item $\omega(N)$ is even
\end{itemize}
\end{proposition}

\begin{proof}
We may assume that $a$ and $h$ are even without loss of generality. Indeed, we may multiply them by two if not.

With probability 1, the majority of $p_j$ do not lie in any specific coset of $(\mathbb{Z}/L\mathbb{Z})^{\times}$ of odd order. For the rest of the proof, therefore, we may assume that for any odd order coset, the majority of the primes lie outside the coset.

Let $g$ be the element $(\square, a)$ in
$$G:=(\mathbb{Z}/L\mathbb{Z})^{\times}\oplus \mathbb{Z}/h\mathbb{Z}.$$
Let $h'$ be the largest odd factor of $h$. We decompose $G$ as $G_1\oplus G_2\oplus G_3\oplus G_4$ where $G_1$ is the maximal subgroup of odd order in $(\mathbb{Z}/L\mathbb{Z})^{\times}$, $G_2 = \mathbb{Z}/h'\mathbb{Z}$, $G_3$ is the 2-Sylow subgroup of $(\mathbb{Z}/L\mathbb{Z})^{\times}$, and $G_4$ is the 2-Sylow subgroup of $\mathbb{Z}/h\mathbb{Z}$. Recall that $q-1$ is square-free for every prime $q$ dividing $L$. Therefore, every element in $G_3$ other than the identity has order $2$. In particular, $\square$ corresponds to $1$ in $G_3$.

Label the elements of $\mathcal{P}^{\star}$ as $p_1, \ldots, p_r$ and let $P$ be their product. We need to bound
\begin{equation}\label{eq:mainbound}
\frac{|G|}{2^r} \sum_{d \mid P} \mathbbm{1}_{G}(d=g) = \frac{1}{2^r} \sum_{d \mid P} \sum_{\chi \in \hat{G}} \overline{\chi(g)} \cdot \chi(d).
\end{equation}
Switching the order of summation, this is equal to
$$\sum_{\substack{\chi_i \in \hat{G}_i \\ i \in \{1,2,3,4\}}}\overline{\chi_1(g)\chi_2(g)\chi_3(g)\chi_4(g)}\prod_{j=1}^r\frac{1+\chi_1(p_j)\chi_2(p_j)\chi_3(p_j)\chi_4(p_j)}{2},$$
where $p_j$ is viewed as the element $(p_j \bmod{L}, 1)$ (according to the original description of $G$). Let $*(\mathbf{\chi})$ be the condition that
$$\chi_1(p_j)\chi_2(p_j)\chi_3(p_j)\chi_4(p_j)=1$$
for the majority of $p_j$. We may break our sum over characters as follows:
$$\sum_{\substack{\chi_i \in \hat{G}_i \\ i \in \{1,2,3,4\} \\ *(\mathbf{\chi})}}\overline{\chi_1(g)\chi_2(g)\chi_3(g)\chi_4(g)}\prod_{j=1}^r\frac{1+\chi_1(p_j)\chi_2(p_j)\chi_3(p_j)\chi_4(p_j)}{2}+O\left(\epsilon|G|\right),$$
where $\epsilon$ is the largest value of
$$\left|\prod_{j=1}^r\frac{1+\chi_1(p_j)\chi_2(p_j)\chi_3(p_j)\chi_4(p_j)}{2}\right|$$
for characters not satisfying $*(\mathbf{\chi})$.

Write $H_1 = G_1$, $H_2 = G_2$, and $H_3 = G_3 \oplus \mathbb{Z}/2\mathbb{Z}$. Since $*(\mathbf{\chi})$ is certainly not true if $\chi_4$ has order greater than two, Equation \eqref{eq:mainbound} is equal to
$$\sum_{\substack{\chi_i \in \hat{H}_i \\ i \in \{1,2,3\} \\ *(\mathbf{\chi})}}\overline{\chi_1(g)\chi_2(g)\chi_3(g)}\prod_{j=1}^r\frac{1+\chi_1(p_j)\chi_2(p_j)\chi_3(p_j)}{2}+O\left(\epsilon|G|\right).$$
By our coset discussion from earlier, if $\chi_1$ is not principal then $*(\mathbf{\chi})$ is not satisfied because $\chi_2(p_j)\chi_3(p_j)$ does not depend on $j$. It also becomes apparent that for $*(\mathbf{\chi})$ to be satisfied, we need $\chi_2$ to be principal. Therefore, Equation \eqref{eq:mainbound} equals
$$\sum_{\substack{\chi_3 \in \hat{H}_3 \\ *(\mathbf{\chi})}}\overline{\chi_3(g)}\prod_{j=1}^r\frac{1+\chi_3(p_j)}{2}+O\left(\epsilon|G|\right).$$
Since $\chi_3(g)=1$, Equation \eqref{eq:mainbound} equals
$$\sum_{\substack{\chi_3 \in \hat{H}_3 \\ *(\mathbf{\chi})}}\prod_{j=1}^r\frac{1+\chi_3(p_j)}{2}+O\left(\epsilon|G|\right).$$
The main term is equal to a constant $|H_3|/2^r$ times the number of divisors of $P$ corresponding to the identity as an element of $H_3$. Proposition \ref{alteredAGPprop} tells us that the number of such divisors is quite large, certainly larger than $2^r/e^{-\lambda(L)^2}$, whereas
$$\epsilon < \left(1-\lambda(L)^{-7}\right)^{r/2} < e^{-\sqrt{r}}.$$ In particular, Equation \eqref{eq:mainbound} has positive real part, which means there has to exist $N$ satisfying our conditions.
\end{proof}

We need one more important tool, of a combinatorial nature.

\begin{lemma}\label{lem:combinatorial}
Let $\mathcal{P}$ be a set of primes and $M$ be an integer such that $p-1 \mid M$ for every $p \in \mathcal{P}$. Let $r$ be a positive integer such that $|\mathcal{P}| \geq r^4$. Suppose $N$ is congruent to 1 mod $M$ and comprised of $O(r)$ distinct primes in $\mathcal{P}$. Finally, suppose $h$ is such that with probability greater than or equal to $1/r$, any $r$-element subset of $\mathcal{P}$ contains a sub-product of length $h$ congruent to 1 mod $M$. Then for every integer in the range $[\omega(N), r^2]$ congruent to $\omega(N)$ mod $h$, there exists a Carmichael number with exactly that many prime factors.
\end{lemma}

\begin{proof}
Let $\widetilde{\mathcal{P}}_0$ consist of the prime factors of $N$. Since $r^3$ is small compared to the size of $\mathcal{P}$, for any $r$ randomly chosen primes from $\mathcal{P}$, the probability of our chosen set of primes meeting $\widetilde{\mathcal{P}}_0$ is less than $1/r$. Therefore, there exists a number congruent to 1 mod $M$ with $h$ prime factors all of which are contained in $\mathcal{P} \setminus \widetilde{\mathcal{P}}_0$.

Define the product of these $h$ primes to be $N_1$.
We add these prime factors to $\widetilde{\mathcal{P}}_0$ to yield $\widetilde{\mathcal{P}}_1$, and we then repeat the process to find $h$ more primes in $\mathcal{P} \setminus \widetilde{\mathcal{P}}_1$ whose product is also 1 mod $M$. By repeating the process as long as we can, we obtain in this way a list of numbers $N_1, \ldots, N_k$, where for all $j \leq k$, $N\prod_{i=1}^{j}N_i$ is a Carmichael number. Therefore every number in the range $[\omega(N), r^2]$ congruent to $\omega(N)$ mod $h$ can be written as $\omega$ evaluated on a Carmichael number.
\end{proof}

Combining Proposition \ref{prop:Lresidue} (with $\square = 1$) and Lemma \ref{lem:combinatorial}, we see that if $h_0$ is odd then every number in the range $[\omega(N), r^2]$ has a Carmichael number with that many prime factors. However, if $h_0$ is even, this is a priori only true for even numbers between $\omega(N)$ and $r^2$. In the next section we will explain how to overcome this obstacle.

\section{Constructing Carmichael Numbers With Odd Numbers of Prime Factors}

We use a very simple trick, fixing 3 as a prime factor at the start. It is always easier to show that there exists a product with specified conditions with an even number of prime factors than an odd number of prime factors, and by fixing a single prime factor to begin with, we use this to our advantage.

To do this, we need to return to Proposition \ref{prop:main}. By applying the methods of the previous three sections to separate values of $L$, $k$, and $\mathcal{P}$, we are ready to prove the following proposition.

\begin{proposition}\label{prop:oddconstruction}
For any $h \leq r$, we can find positive integers $x$, $L_1$, $L_2$, $k_1$, $k_2$, $A_1(h)$, $A_2(h)$ and sets $\mathcal{P}_1(h)$, $\mathcal{P}_2(h)$ satisfying all the conditions of Proposition \ref{prop:main}, with the additional property that we can pick $A_1(h)$ and $A_2(h)$ so that $\omega(A_1(h))$ and $\omega(A_2(h))$ are congruent to any two residues mod $h$ we like, as long as they correspond to arithmetic progressions containing even numbers.
\end{proposition}

\begin{proof}
The construction of $x$, $L_1$, $L_2$, $k_1$, and $k_2$ has already been done, applying the $\mathcal{Q}$-splitting of \cite{L:2025} to our work in the previous sections. Let $A$ be one of the integers guaranteed by Proposition \ref{prop:Lresidue}, with $\square$ corresponding to $1 \bmod{L_1}$ and $3^{-1} \bmod{L_2}$. Let $P$ be the largest square-free integer congruent to 1 mod $L_1L_2$ whose prime factors lie in $\mathcal{P}_1$ and which is prime to $A$. Finally, apply Proposition \ref{alteredAGPprop} to the sequence of prime factors of $AP$ to construct a set $\mathcal{A}_1$ whose elements are congruent to 1 mod $L_1$ and $3^{-1}$ mod $L_2$. Taking $t$ to be twice the Davenport constant of $(\mathbb{Z}/L_1L_2\mathbb{Z})^{\times}$, we have in particular that
$$\log |\mathcal{A}_1| \gg t \log x,$$
recalling Equation \eqref{eq:Pbound}.
By Equation (11) of \cite{L:2025}, the resulting bound
$$\frac{\log \max_{A \in \mathcal{A}_1} A}{\log |\mathcal{A}_1|} = O(1)$$
implies that for almost all choices of $k_2$, there exists $A_1 \in \mathcal{A}_1$ such that $3A_1$ has odd order mod $k_2$. We subtract the prime factors of $A_1$ to obtain the final set $\mathcal{P}_1$. Of course, we can then perform the same operation with indices swapped.
\end{proof}

\begin{remark}
This goes beyond the original construction in \cite{L:2025} because here $h$, unlike $\ell_0$, is not an $O(1)$ quantity.
\end{remark}

Let $L = L_1L_2$. To finish off the argument, we specialize Lemma 15 from \cite{L:2025} to the case $w_0 = 0$ and $\ell_0 = 2$.

\begin{lemma}\label{hittarget}
Subject to the conditions of Proposition \ref{prop:main}, with 
$$G := (\mathbb{Z}/L\mathbb{Z})^{*}\oplus (\mathbb{Z}/k_2\mathbb{Z})^{*}\oplus \mathbb{Z}/2\mathbb{Z},$$
if $g \in G$ is 1 mod $L$, odd order mod $k_2$, and $0$ mod $2$, then for almost all $r$ element subsets $\mathcal{P}^{\star} \subset \mathcal{P}_1$, there exists a product of distinct elements in $\mathcal{P}^{\star}$ equal to $g$ as an element of $G$.
\end{lemma}

Take $g$ to be 1 mod $L$, 1 mod $k_2$, and $0$ mod $2$. With probability 1, a subset of $\mathcal{P}_1$ of size $r$ has a sub-product congruent to 1 mod $Lk_2$ with an even number of prime factors. Therefore, there exists $h \leq r$ such that with probability at least $1/r$, a subset of $\mathcal{P}_1$ of size $r$ has a sub-product of length $h$ congruent to 1 mod $Lk_2$.

For this value of $h$, choose $A_1$ and $A_2$ in Proposition \ref{prop:oddconstruction} corresponding respectively to the residues $a$ and $0$ mod $h$, where we take $a$ to be an arbitrary even number.

Now take $g$ to be 1 mod $L$, $(3A_1)^{-1}$ mod $k_2$, and $0$ mod $2$. Let $\Pi_1$ be one of the products guaranteed by Lemma \ref{hittarget}. Now flipping the indices, this lemma can also produce $\Pi_2$ which is congruent to 1 mod $L$ and $(3A_2)^{-1}$ mod $k_1$, while having an even number of prime factors. 

Now take $N = 3A_1A_2\Pi_1\Pi_2$. Note that $N$ is congruent to 1 mod $L$, $k_1$, and $k_2$, meaning in particular that $N$ is congruent to 1 mod $Lk_1$ and $Lk_2$. Moreover, $\omega(N)$ is congruent to $a+1+\omega(\Pi_1\Pi_2)$ mod $h$.

By Lemma \ref{lem:combinatorial}, every number in the range $[\omega(N), r^2]$ congruent to $a+1+\omega(\Pi_1\Pi_2)$ mod $h$ can be written as $\omega$ evaluated on a Carmichael number.

Since $a$ was an arbitrary even number, every odd number in the range $[\omega(N), r^2]$ has the property that it can be written as $\omega$ evaluated on a Carmichael number. Combining this with our work in the previous section, every number in the range $[\omega(N), r^2]$ has this property.

Clearly $\omega(N) \leq 4r+1$. Recalling the definition of $r$, for every number in the range $[4y_5^{\sqrt{y_2 w}}+1, y_5^{2\sqrt{y_2 w}}]$, there exists a Carmichael number with exactly that many prime factors. Since
$$y_2 = y_5^{1/6-1/500} \quad \text{and} \quad w = \lfloor y_5^{(1/6-1/500) + (1/1000)(1/24-1/500)}\rfloor,$$
we may immediately deduce Theorem \ref{thm:main} by repeating the construction in this paper, incrementing $y_5$ by one each time. Indeed, it is clear that these intervals will overlap and their union will contain all sufficiently large integers.

\end{document}